\documentclass[preprint]{elsarticle}

\usepackage{amssymb}
\usepackage{amsmath}
\usepackage{amsthm}
\usepackage{hyperref}
\usepackage{algorithm}
\usepackage{algorithmic}

\newtheorem{theorem}{Theorem}[section]
\newtheorem{lemma}[theorem]{Lemma}
\newtheorem{proposition}[theorem]{Proposition}
\newtheorem{corollary}[theorem]{Corollary}

\theoremstyle{definition}
\newtheorem{definition}[theorem]{Definition}
\newtheorem{example}[theorem]{Example}

\theoremstyle{remark}
\newtheorem{remark}[theorem]{Remark}

\numberwithin{equation}{section}

\if{
\newtheorem{theorem}{Theorem}[section]

\newtheorem{proposition}{Proposition}[section]
\newtheorem{corollary}{Corollary}[section]
\newtheorem{example}{Example}[section]

\newtheorem{definition}{Definition}[section]
}\fi

\newcommand{\cle}{\preceq}
\newcommand{\opl}{{\oplus}}

\newcommand{\rmin}{\mathbf{R}_{\min}}
\newcommand{\rmax}{\mathbf{R}_{\max}}
\newcommand{\rmaxh}{\hat{\mathbf{R}}_{\max}}
\newcommand{\smaxmin}{S_{\mathrm{max,min}}}
\newcommand{\zmax}{\mathbf{Z}_{\max}}
\newcommand{\zmaxh}{\hat{\mathbf{Z}}_{\max}}
\newcommand{\Z}{\mathbf{Z}}

\newcommand{\suplim}{\sup\limits}
\newcommand{\sumlim}{\sum\limits}
\newcommand{\maxlim}{\max\limits}
\newcommand{\pd}[2]{\dfrac{\partial#1}{\partial#2}}

\newcommand{\alg}{\mathcal{A}}

\newcommand{\cA}{{\mathcal A}}

\newcommand{\cP}{{\mathcal P}}

\newcommand{\0}{\mathbf{0}}
\newcommand{\1}{\mathbf{1}}

\newcommand{\rset}{\mathbf{R}}

\def\C{\mathbf C}
\def\R{\mathbf R}
\def\Rp{{\mathbf R}_+}

\def\x{\mathbf{x}}
\def\intx{\x}
\def\y{\mathbf{y}}
\def\inty{\y}
\def\podx{\underline{\x}}
\def\nadx{\overline{\x}}
\def\pody{\underline{\y}}
\def\nady{\overline{\y}}
\def\Mat{\operatorname{Mat}}

\begin{document}

\title{Idempotent and tropical mathematics. \\ Complexity of algorithms and interval analysis}
\author{G. L. Litvinov\fnref{fn1}}

\address{Institute for Information Transmission Problems of Russian Acad. Sci.}

\address{Nagornaya, 27--4--72, Moscow 117186 Russia, tel. 7-499-1273804}
\ead{glitvinov@gmail.com}

\fntext[fn1]{This work is  supported by the RFBR grant 12-01-00886-a and joint RFBR/CNRS grant 11-01-93106-a}

\begin{abstract}
A very brief introduction to tropical and
idempotent mathematics is presented. Tropical mathematics can be treated as a result of a dequantization of
the traditional mathematics as the Planck constant tends to zero taking
imaginary values. In the framework of idempotent mathematics usually constructions and
algorithms are more simple with respect to their traditional analogs. We
especially examine algorithms of tropical/idempotent mathematics
generated by a collection of basic semiring (or semifield) operations
and other "good" operations. Every algorithm of this type has an interval
version. The complexity of this interval version coincides with the complexity
of the initial algorithm. The interval version of an algorithm of this type gives
exact interval estimates for the corresponding output data. Algorithms of
linear algebra over idempotent and semirings are examined. In this case,
basic algorithms are polynomial as well as their interval versions. This
situation is very different from the traditional linear algebra, where basic
algorithms are polynomial but the corresponding interval versions are
NP-hard and interval estimates are not exact.
\end{abstract}


\maketitle

\section{Introduction}
\label{sec:intro}

Tropical mathematics can be treated  as  a result of a
dequantization of the traditional mathematics as  the Planck
constant  tends to zero  taking imaginary values. This kind of
dequantization is known as the Maslov dequantization and it leads
to a mathematics over tropical algebras like the max-plus algebra.
The so-called idempotent dequantization is a generalization of the
Maslov dequantization. The idempotent dequantization leads to
mathematics over idempotent semirings (exact definitions see below
in sections 2 and 3). For  example,
the field of real or complex numbers can be treated as a quantum
object whereas idempotent semirings can be examined  as
"classical" or "semiclassical" objects  (a semiring is called
idempotent  if the  semiring addition is idempotent, i.e. $x
\oplus x = x$), see~\cite{Li2005,Li2011,LiMa95,LiMa96,LiMa98}.

Tropical algebras are idempotent semirings (and semifields). Thus tropical mathematics is a part of idempotent mathematics. Tropical
algebraic geometry can be treated as a result of the Maslov
dequantization applied to the traditional algebraic geometry (O.
Viro, G. Mikhalkin), see,
e.g.,~\cite{Iten2007,Mi2005,Mi2006,Vir2000,Vir2002,Vir2008}. There
are interesting relations and applications to the traditional
convex geometry.

In the spirit of N.Bohr's correspondence principle there is a (heuristic)  correspondence
between important, useful, and interesting  constructions  and  results over fields and
similar results  over  idempotent  semirings.  A systematic application of this
correspondence  principle (which is a basic paradigm in idempotent/tropical mathematcs)
leads  to  a variety  of theoretical and applied
results~\cite{Li2005,Li2011,LiMa95,LiMa96,LiMa98,LiMa2005,LiMaRS2011}, see Fig.1.

\begin{figure}
\centering
\includegraphics[width=0.9\textwidth]{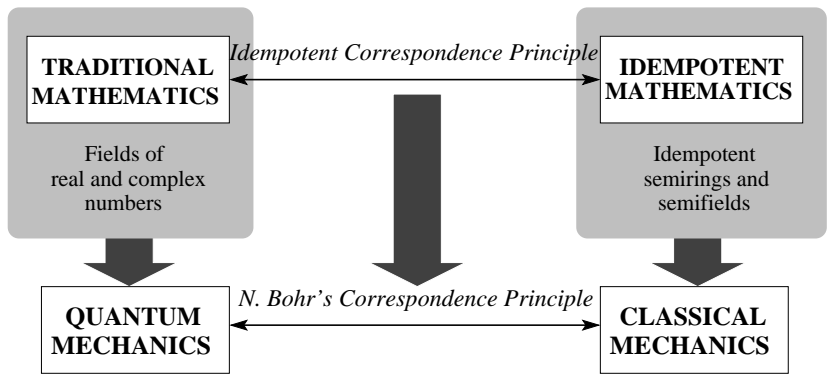}
\caption{Relations between idempotent and traditional mathematics.}
\end{figure}

The history of the subject is discussed, e.g., in~\cite{Li2005}.
There is a large list of references.

In the framework of idempotent mathematics usually constructions and
algorithms are more simple with respect to their traditional analogs (however,
there exist NP-hard problems in tropical linear algebra). We
especially examine algorithms of tropical/idempotent mathematics
generated by a collection of basic semiring (or semifield) operations
and other "good" operations. Every algorithm of this type has an interval
version. The complexity of this interval version coincides with the complexity
of the initial algorithm. The interval version of an algorithm of this type gives
exact interval estimates for the corresponding output data. Algorithms of
linear algebra over idempotent and semirings are examined. In this case,
basic algorithms are polynomial as well as their interval versions. This
situation is very different from the traditional linear algebra, where basic
algorithms are polynomial but the corresponding interval versions are
NP-hard and interval estimates are not exact.
\medskip

\section{The Maslov dequantization}

Let $\R$ and $\C$ be the fields of real and complex numbers. The
so-called max-plus algebra $\R_{\max}= \R\cup\{-\infty\}$ is
defined by the operations $x\oplus y=\max\{x, y\}$ and $x\odot y=
x+y$.

The max-plus algebra can be treated as a result of the {\it Maslov
dequantization} of the semifield $\R_+$ of all nonnegative
numbers with the usual arithmetics. The change of variables
\begin{eqnarray*}
x\mapsto u=h\log x,
\end{eqnarray*}
where $h>0$, defines a map $\Phi_h\colon \R_+\to
\R\cup\{-\infty\}$. This logarithmic transform was used by many authors. Let the addition and multiplication
operations be mapped from \markboth{G.L. Litvinov}{Dequantization...} $\R_+$ to $\R\cup\{-\infty\}$ by $\Phi_h$, i.e.\
let
\begin{eqnarray*}
u\oplus_h v = h \log({\mbox{exp}}(u/h)+{\mbox{exp}}(v/h)),\quad u\odot v= u+ v,\\
\mathbf{0}=-\infty = \Phi_h(0),\quad \mathbf{1}= 0 = \Phi_h(1).
\end{eqnarray*}


It can easily be checked that $u\oplus_h v\to \max\{u, v\}$ as
$h\to 0$. Thus we get the semifield $\R_{\max}$ (i.e.\ the
max-plus algebra) with zero $\mathbf{0}= -\infty$ and unit
$\mathbf{1}=0$ as a result of this deformation of the algebraic
structure~in~$\R_+$.

The semifield $\R_{\max}$ is a typical example of an {\it
 idempotent semiring}; this is a semiring with idempotent addition, i.e.,
 $x\oplus x = x$ for arbitrary element
 $x$ of this semiring.

 The semifield $\R_{\max}$ is also called a \emph{tropical
 algebra}.The semifield $\R^{(h)}=\Phi_h(\R_+)$ with operations
 $\oplus_h$ and $\odot$ (i.e.$+$) is called a \emph{subtropical
 algebra}.

 The semifield $\R_{\min}=\R\cup\{+\infty\}$
with operations $\oplus={\min}$ and $\odot=+$
$(\mathbf{0}=+\infty, \mathbf{1}=0)$ is isomorphic to $\R_{\max}$.

The analogy with quantization is obvious; the parameter $h$ plays
the role of the Planck constant. The map $x\mapsto|x|$ and the
Maslov dequantization for $\R_+$ give us a natural transition from
the field $\C$ (or $\R$) to the max-plus algebra $\R_{\max}$. {\it
We will also call this transition the Maslov dequantization}. In fact
the Maslov dequantization corresponds to the usual Schr\"odinger
dequantization but for  imaginary values of the Planck constant (see below).
The transition from numerical fields to the max-plus algebra
$\R_{\max}$ (or similar semifields) in mathematical constructions
and results generates the so called {\it tropical mathematics}.
The so-called {\it idempotent dequantization} is a generalization
of the Maslov dequantization; this is the transition from basic fields to idempotent semirings in mathematical constructions and results without any deformation. The idempotent dequantization generates
the so-called \emph{idempotent mathematics}, i.e. mathematics over
idempotent semifields and semirings. Recently new versions of the Maslov dequantization appeared, see, e.g. \cite{Vir2010}.

{\bf Remark.} The term 'tropical' appeared in~\cite{Sim88} for a
discrete version of the max-plus algebra (as a suggestion of
Christian Choffrut). On the other hand V.P. Maslov used this term in
80s in his talks and works on economical applications of his
idempotent analysis (related to colonial politics). For the most
part of modern authors, 'tropical' means 'over $\R_{\max}$ (or
$\R_{\min}$)' and tropical algebras are $\R_{\max}$ and $\R_{\min}$.
The terms 'max-plus', 'max-algebra' and 'min-plus' are often used in
the same sense.
\medskip

\section{Semirings and semifields}

Consider a set $S$ equipped with two algebraic operations: {\it
addition} $\oplus$ and {\it multiplication} $\odot$. It is a {\it
semiring} if the following conditions are satisfied:
\begin{itemize}
\item the addition $\oplus$ and the multiplication $\odot$ are
associative; \item the addition $\oplus$ is commutative; \item the
multiplication $\odot$ is distributive with respect to the
addition $\oplus$:
\[x\odot(y\oplus z)=(x\odot y)\oplus(x\odot z)\]
and
\[(x\oplus y)\odot z=(x\odot z)\oplus(y\odot z)\]
for all $x,y,z\in S$.
\end{itemize}
A {\it unity} of a semiring $S$ is an element $\1\in S$ such that
$\1\odot x=x\odot\1=x$ for all $x\in S$. A {\it zero} of a
semiring $S$ is an element (if it exists) $\0\in S$ such that $\0\neq\1$ and
$\0\oplus x=x$, $\0\odot x=x\odot \0=\0$ for all $x\in S$. A
semiring $S$ is called an {\it idempotent semiring} if $x\oplus
x=x$ for all $x\in S$. A semiring $S$ with a neutral element $\1$ is called a {\it semifield} if every nonzero element of
$S$ is invertible with respect to the multiplication. The theory
of semirings and semifields is treated, e.g., in~\cite{Gol99}.
\medskip

\section{Idempotent analysis}

Idempotent analysis deals with functions taking their values in
an idempotent semiring and the corresponding function spaces.
Idempotent analysis was initially constructed by V.~P.~Maslov and
his collaborators and then developed by many authors. The subject
is presented in the book of V.~N.~Kolokoltsov and
V.~P.~Maslov~\cite{KoMa97} (a version of this book in Russian was
published in 1994).

Let $S$ be an arbitrary semiring with idempotent addition $\oplus$
(which is always assumed to be commutative), multiplication
$\odot$, and unit $\1$. The set $S$ is supplied with
the {\it standard partial order\/}~$\cle$: by definition, $a \cle
b$ if and only if $a \oplus b = b$. If the zero element exists, then all elements of $S$ are
nonnegative: $\0 \cle$ $a$ for all $a \in S$. Due to the existence
of this order, idempotent analysis is closely related to the
lattice theory, theory of vector lattices, and theory of ordered
spaces. Moreover, this partial order allows to model a number of
basic ``topological'' concepts and results of idempotent analysis
at the purely algebraic level; this line of reasoning was examined
systematically in~\cite{Li2005}--~\cite{LiSh2007}
and~\cite{CoGaQu2004}.

Calculus deals mainly with functions whose values are numbers. The
idempotent analog of a numerical function is a map $X \to S$,
where $X$ is an arbitrary set and $S$ is an idempotent semiring.
Functions with values in $S$ can be added, multiplied by each
other, and multiplied by elements of $S$ pointwise.

The idempotent analog of a linear functional space is a set of
$S$-valued functions that is closed under addition of functions
and multiplication of functions by elements of $S$, or an
$S$-semimodule. Consider, e.g., the $S$-semimodule $B(X, S)$ of
all functions $X \to S$ that are bounded in the sense of the
standard order on $S$.

If $S = \rmax$, then the idempotent analog of integration is
defined by the formula
$$
I(\varphi) = \int_X^{\oplus} \varphi (x)\, dx     = \sup_{x\in X}
\varphi (x),\eqno{(1)}
$$
where $\varphi \in B(X, S)$. Indeed, a Riemann sum of the form
$\sumlim_i \varphi(x_i) \cdot \sigma_i$ corresponds to the
expression $\bigoplus\limits_i \varphi(x_i) \odot \sigma_i =
\maxlim_i \{\varphi(x_i) + \sigma_i\}$, which tends to the
right-hand side of~(1) as $\sigma_i \to 0$. Of course, this is a
purely heuristic argument.

Formula~(1) defines the \emph{idempotent} (or \emph{Maslov})
\emph{integral} not only for functions taking values in $\rmax$,
but also in the general case when any of bounded (from above)
subsets of~$S$ has the least upper bound.

An \emph{idempotent} (or \emph{Maslov}) \emph{measure} on $X$ is
defined by the formula $m_{\psi}(Y) = \suplim_{x \in Y} \psi(x)$, where $\psi
\in B(X,S)$ is a fixed function. The integral with respect to this measure is defined
by the formula
$$
   I_{\psi}(\varphi)
    = \int^{\oplus}_X \varphi(x)\, dm_{\psi}
    = \int_X^{\oplus} \varphi(x) \odot \psi(x)\, dx
    = \sup_{x\in X} (\varphi (x) \odot \psi(x)).
    \eqno{(2)}
$$

Obviously, if $S = \rmin$, then the standard order is
opposite to the conventional order $\le$, so in this case
equation~(2) assumes the form
$$
   \int^{\oplus}_X \varphi(x)\, dm_{\psi}
    = \int_X^{\oplus} \varphi(x) \odot \psi(x)\, dx
    = \inf_{x\in X} (\varphi (x) \odot \psi(x)),
$$
where $\inf$ is understood in the sense of the conventional order
$\le$.
\medskip

\section{The superposition principle and linear problems}

Basic equations of quantum theory are linear; this is the
superposition principle in quantum mechanics. The Hamilton--Jacobi
equation, the basic equation of classical mechanics, is nonlinear
in the conventional sense. However, it is linear over the
semirings $\rmax$ and $\rmin$. Similarly, different versions of
the Bellman equation, the basic equation of optimization theory,
are linear over suitable idempotent semirings; this is
V.~P.~Maslov's idempotent superposition principle, see
\cite{Mas86,Mas87a,Mas87b}. For instance, the finite-dimensional
stationary Bellman equation can be written in the form $X = H
\odot X \oplus F$, where $X$, $H$, $F$ are matrices with
coefficients in an idempotent semiring $S$ and the unknown matrix
$X$ is determined by $H$ and $F$~\cite{Ca71,Ca79}. In
particular, standard problems of dynamic programming and the
well-known shortest path problem correspond to the cases $S =
\rmax$ and $S =\rmin$, respectively. It is known that principal
optimization algorithms for finite graphs correspond to standard
methods for solving systems of linear equations of this type
(i.e., over semirings). Specifically, Bellman's shortest path
algorithm corresponds to a version of Jacobi's algorithm, Ford's
algorithm corresponds to the Gauss--Seidel iterative scheme,
etc.~\cite{Ca71,Ca79}.

The linearity of the Hamilton--Jacobi equation over $\rmin$ and
$\rmax$, which is the result of the Maslov dequantization of the
Schr{\"o}\-din\-ger equation, is closely related to the
(conventional) linearity of the Schr{\"o}\-din\-ger equation and
can be deduced from this linearity. Thus, it is possible to borrow
standard ideas and methods of linear analysis and apply them to a
new area.

Consider a classical dynamical system specified by the Hamiltonian
$$
   H = H(p,x) = \sum_{i=1}^N \frac{p^2_i}{2m_i} + V(x),
$$
where $x = (x_1, \dots, x_N)$ are generalized coordinates, $p =
(p_1, \dots, p_N)$ are generalized momenta, $m_i$ are generalized
masses, and $V(x)$ is the potential. In this case the Lagrangian
$L(x, \dot x, t)$ has the form
$$
   L(x, \dot x, t)
    = \sum^N_{i=1} m_i \frac{\dot x_i^2}2 - V(x),
$$
where $\dot x = (\dot x_1, \dots, \dot x_N)$, $\dot x_i = dx_i /
dt$. The value function $S(x,t)$ of the action functional has the
form
$$
   S = \int^t_{t_0} L(x(t), \dot x(t), t)\, dt,
$$
where the integration is performed along the factual trajectory of
the system.  The classical equations of motion are derived as the
stationarity conditions for the action functional (the Hamilton
principle, or the least action principle).

For fixed values of $t$ and $t_0$ and arbitrary trajectories
$x(t)$, the action functional $S=S(x(t))$ can be considered as a
function taking the set of curves (trajectories) to the set of
real numbers which can be treated as elements of  $\rmin$. In this
case the minimum of the action functional can be viewed as the
Maslov integral of this function over the set of trajectories or
an idempotent analog of the Euclidean version of the Feynman path
integral. The minimum of the action functional corresponds to the
maximum of $e^{-S}$, i.e. idempotent integral
$\int^{\oplus}_{\{paths\}} e^{-S(x(t))} D\{x(t)\}$ with respect to
the max-plus algebra $\rset_{\max}$. Thus the least action
principle can be considered as an idempotent version of the
well-known Feynman approach to quantum mechanics.  The
representation of a solution to the Schr{\"o}\-din\-ger equation
in terms of the Feynman integral corresponds to the
Lax--Ole\u{\i}nik solution formula for the Hamilton--Jacobi
equation.

Since $\partial S/\partial x_i = p_i$, $\partial S/\partial t =
-H(p,x)$, the following Hamilton--Jacobi equation holds:
$$
   \pd{S}{t} + H \left(\pd{S}{x_i}, x_i\right)= 0.\eqno{(3)}
$$

Quantization leads to the Schr\"odinger equation
$$
   -\frac{\hbar}i \pd{\psi}{t}= \widehat H \psi = H(\hat p_i, \hat x_i)\psi,
   \eqno{(4)}
$$
where $\psi = \psi(x,t)$ is the wave function, i.e., a
time-dependent element of the Hilbert space $L^2(\rset^N)$, and
$\widehat H$ is the energy operator obtained by substitution of
the momentum operators $\widehat p_i = {\hbar \over i}{\partial
\over \partial x_i}$ and the coordinate operators $\widehat x_i
\colon \psi \mapsto x_i\psi$ for the variables $p_i$ and $x_i$ in
the Hamiltonian function, respectively. This equation is linear in
the conventional sense (the quantum superposition principle). The
standard procedure of limit transition from the Schr\"odinger
equation to the Hamilton--Jacobi equation is to use the following
ansatz for the wave function:  $\psi(x,t) = a(x,t)
e^{iS(x,t)/\hbar}$, and to keep only the leading order as $\hbar
\to 0$ (the `semiclassical' limit).

Instead of doing this, we switch to imaginary values of the Planck
constant $\hbar$ by the substitution $h = i\hbar$, assuming $h >
0$. Thus the Schr\"odinger equation~(4) turns to an analog of the
heat equation:
$$
   h\pd{u}{t} = H\left(-h\frac{\partial}{\partial x_i}, \hat x_i\right) u,
   \eqno{(5)}
$$
where the real-valued function $u$ corresponds to the wave
function $\psi$. A similar idea (the switch to imaginary time) is
used in the Euclidean quantum field theory; let us remember that
time and energy are dual quantities.

Linearity of equation~(4) implies linearity of equation~(5). Thus
if $u_1$ and $u_2$ are solutions of~(5), then so is their linear
combination
$$
   u = \lambda_1 u_1 + \lambda_2 u_2.\eqno{(6)}
$$

Let $S = h \ln u$ or $u = e^{S/h}$ as in Section 2 above. It can
easily be checked that equation~(5) thus turns to
$$
   \pd{S}{t}= V(x) + \sum^N_{i=1} \frac1{2m_i}\left(\pd{S}{x_i}\right)^2
    + h\sum^n_{i=1}\frac1{2m_i}\frac{\partial^2 S}{\partial x^2_i}.
   \eqno{(7)}
$$
Thus we have a transition from (4) to (7) by means of the change of
variables $\psi = e^{S/h}$. Note that $|\psi| = e^{ReS/h}$ , where
Re$S$ is the real part  of $S$. Now let us consider $S$ as a real
variable. The equation (7) is nonlinear in the conventional sense.
However, if $S_1$ and $S_2$ are its solutions, then so is the
function
$$
   S = \lambda_1 \odot S_1 \opl_h \lambda_2\odot S_2
$$
obtained from~(6) by means of our substitution $S = h \ln u$. Here
the generalized multiplication $\odot$ coincides with the ordinary
addition and the generalized addition $\opl_h$ is the image of the
conventional addition under the above change of variables.  As $h
\to 0$, we obtain the operations of the idempotent semiring
$\rmax$, i.e., $\oplus = \max$ and $\odot = +$, and equation~(7)
turns to the Hamilton--Jacobi equation~(3), since the third term
in the right-hand side of equation~(7) vanishes.

Thus it is natural to consider the limit function $S = \lambda_1
\odot S_1 \oplus \lambda_2 \odot S_2$ as a solution of the
Hamilton--Jacobi equation and to expect that this equation can be
treated as linear over $\rmax$. This argument (clearly, a
heuristic one) can be extended to equations of a more general
form. For a rigorous treatment of (semiring) linearity for these
equations see, e.g., \cite{KoMa97,LiMa2005}. Notice that if
$h$ is changed to $-h$, then we have that the resulting
Hamilton--Jacobi equation is linear over $\rmin$.

The idempotent superposition principle indicates that there exist
important nonlinear (in the traditional sense) problems that are
linear over idempotent semirings. The idempotent linear functional
analysis (see~\cite{KoMa97,Li2011,LiMa2005,LiMaSh98,LiMaSh99,LiMaSh2001,LiMaSh2002,LiSh2002,LiSh2007,MasSam2002,CoGaQu2004,GM:10})
is a natural tool for investigation of those nonlinear infinite-dimensional problems that possess this
property. In practice infinite-dimensional problems can be approximated by finite-dimensional problems.
So algorithms of idempotent linear algebras are especially important (because of the superposition principle!).
Below some examples are examined.

\medskip

\section{Applications}

There are very many important applications of tropical/idempotent mathematics (and especially the correspondence and
superposition principles) including optimization and control, algebraic geometry, dynamic programming,
differential equations, mathematical biology, mathematical physics and chemistry, transport and energoenergetic netwoks, interval analysis, mathematical
economics, game theory, computer technology etc., see, e.g.~\cite{Ca71,Ca79,Gun98a,KoMa97,Li2005,LiMa95,LiMa96,LiMa2005,LiSer2007,LiMaRS2011,LRSS,LiSer2009,
Mas86,Mas87a,MasSam2002,Mi2005,Mi2006,NGVR,Vir2000,Vir2008,Vir2010}. Applications of the idempotent correspondence principles to software and hardware
design are examined, e.g. in \cite{LiMa95,LiMa96,LiMaRS2011}. Some applications are discucced in~\cite{LMRS,LiSer2007}.
\medskip

\section{Positive semirings and basic operations}

\subsection{Some definitions}

Let the semiring~$S$ be partially ordered (see, e.g. \cite{Bir,Gol99} and Subsection 9.2 below) by a relation~$\preceq$ such that $\0$ is
the least element and the inequality $x \preceq y$ implies that $x \oplus z
\preceq y \oplus z$, $x \odot z \preceq y \odot z$, and~$z \odot x \preceq z \odot
y$ for all $x, y, z \in S$; in this case the semiring~$S$ is called
\emph{positive} (see, e.g., \cite{Gol99}).

Recall that a semiring~$S$ is called \emph{idempotent} if $x \oplus x = x$ for all
$x \in S$. 
In this case the
addition~$\oplus$ defines a \emph{canonical (or standard)
partial order}~$\preceq$ on the semiring~$S$ by the rule: $x\preceq y$ iff $x\oplus y=y$. It is easy to prove that any idempotent semiring is
positive with respect to this order. Note also that $x \oplus y =
\sup\{x,y\}$ with respect to the canonical order. In the sequel, we shall
assume that all idempotent semirings are ordered by the canonical partial
order relation.

We shall say that a positive (e.g., idempotent) semiring $S$ is \emph{complete} if it is complete  as an ordered set. This means that for every subset $T\subset S$ there exist elements $\sup T\in S$ and $\inf T\in S$.

The most well-known and important examples of positive semirings are
``numerical'' semirings consisting of (a subset of) real numbers and ordered
by the usual linear order $\leq$ on~$\rset$: the semiring~$\rset_+$
with the usual operations $\oplus = +$, $\odot = \cdot$ and neutral
elements $\0 = 0$, $\1 = 1$, the semiring~$\rmax = \rset \cup \{-\infty\}$
with the operations $\oplus = \max$, $\odot = +$ and neutral elements $\0 =
-\infty$, $\1 = 0$, the semiring $\rmaxh = \rmax \cup \{\infty\}$,
where $x \preceq \infty$, $x \oplus \infty = \infty$ for all $x$, $x \odot
\infty = \infty \odot x = \infty$ if $x \neq \0$, and $\0 \odot \infty =
\infty \odot \0$, the semirings $\zmax=\Z\cup\{-\infty\}$ and
$\zmaxh=\zmax\cup\{+\infty\}$ (subsemirings in $\rmax$ and $\rmaxh$),
and the semiring~$\smaxmin^{[a,b]} = [a, b]$, where
$-\infty \leq a < b \leq +\infty$, with the operations $\oplus =
\max$, $\odot = \min$ and neutral elements $\0 = a$, $\1 = b$.  The
semirings~$\rmax$, $\rmaxh$, $\zmax$, $\zmaxh$ and~$\smaxmin^{[a,b]} = [a, b]$ are
idempotent. The semirings $\rmaxh$, $\zmaxh$,  $\smaxmin^{[a,b]}$, $\widehat{\textbf{R}}_+=\textbf{R}_+\bigcup\{\infty\}$ are complete.
Remind that every partially ordered set can be imbedded to its completion (a minimal complete set containing the initial one).
We shall say that all these semirings (as well as algebras isomorphic to them) are {\it basic numerical semiring}. These semirings are complete or their completions are complete semirings.

Denote by $\Mat_{mn}(S)$ a set of all matrices $A = (a_{ij})$
with $m$~rows and $n$~columns whose coefficients belong to a semiring~$S$.
The sum $A \oplus B$ of matrices $A, B \in \Mat_{mn}(S)$ and the product
$AB$ of matrices $A \in \Mat_{lm}(S)$ and $B \in \Mat_{mn}(S)$ are defined
according to the usual rules of linear algebra:
$A\oplus B=(a_{ij} \oplus b_{ij})\in \mathrm{Mat}_{mn}(S)$ and
$$
AB=\left(\bigoplus_{k=1}^m a_{ij}\odot b_{kj}\right)\in\Mat_{ln}(S),
$$
where $A\in \Mat_{lm}(S)$ and $B\in\Mat_{mn}(S)$.
Note that we write $AB$ instead of $A\odot B$.

If the semiring~$S$ is
positive, then the set $\Mat_{mn}(S)$ is ordered by the relation $A =
(a_{ij}) \preceq B = (b_{ij})$ iff $a_{ij} \preceq b_{ij}$ in~$S$ for all $1
\leq i \leq m$, $1 \leq j \leq n$.

The matrix multiplication is consistent with the order~$\preceq$ in the
following sense: if $A, A' \in \Mat_{lm}(S)$, $B, B' \in \Mat_{mn}(S)$ and
$A \preceq A'$, $B \preceq B'$, then $AB \preceq A'B'$ in $\Mat_{ln}(S)$. The set
$\Mat_{nn}(S)$ of square $(n \times n)$ matrices over a [positive,
idempotent] semiring~$S$ forms a [positive, idempotent] semiring with a
zero element $O = (o_{ij})$, where $o_{ij} = \0$, $1 \leq i, j
\leq n$, and a unit element $I = (\delta_{ij})$, where $\delta_{ij} =
\1$ if $i = j$ and $\delta_{ij} = \0$ otherwise.

The set $\Mat_{nn}$ is an example of a noncommutative semiring if $n>1$.


\subsection{Closure operations}

Let a positive semiring~$S$ be endowed with a partial unary \emph{closure
operation or Kleene star operation}~$*$ such that $a \preceq b$ implies $a^* \preceq b^*$ and $a^* = \1
\oplus (a^* \odot a) = \1 \oplus (a \odot a^*)$ on its domain of
definition. In particular, $\0^* = \1$ by definition.

These axioms imply
that $a^* = \1 \oplus a \oplus a^2 \oplus \dots \oplus (a^* \odot a^n)$ if
$n \geq 1$. Thus $x^*$ can be considered as a `regularized sum' of the
series $a^* = \1 \oplus a \oplus a^2 \oplus \dots$.  In a positive
semiring, provided that it is closed under taking bounded ordered sup-operations
and the operations $\oplus$ and $\odot$ distribute over such sup-operations
we can define
\begin{equation}
\label{cldef-gen}
a^*:=\sup_{k\geq 0} \1\oplus a\oplus \ldots\oplus a^k,
\end{equation}
if the sequence on the r.h.s. is bounded. In this case $a^*$ is the {\bf least solution} of the equations $x=ax\oplus\1$ and $x=xa\oplus\1$, and more generally $a^*b$ is the
the least solution
of the {\it Bellman equations} $x=ax\oplus b$ and $x=xa\oplus b$. So if $S$ is complete, then the closure operation is well-defined for every element $x\in S$.

In the case of idempotent addition~\eqref{cldef-gen}
becomes particularly nice:
\begin{equation}
\label{aldef-id}
a^*=\bigoplus_{i\geq 0} a^i=\sup_{i\geq 0} a^i.
\end{equation}

In numerical semirings the operation~$*$ is usually very easy to implement:
$x^* = (1-x)^{-1}$ if $x<1$ in $\rset_+$, or $\widehat{\textbf{R}}_+$ and $x^*=\infty$ if $x\geq 1$ in $\widehat{\textbf{R}}_+$; $x^* = \1$ if $x \preceq \1$ in $\rmax$
and $\rmaxh$, $x^* = \infty$ if $x \succ \1$ in $\rmaxh$, $x^* = \1$
for all $x$ in $\smaxmin^{[a,b]}$. In all other cases $x^*$ is undefined.

The closure operation in matrix semirings over a positive semiring~$S$ can
be defined inductively: $A^*
= (a_{11})^* = (a^*_{11})$ in $\Mat_{11}(S)$ and for any integer $n > 1$
and any matrix
$$
   A = \begin{pmatrix} A_{11}& A_{12}\\ A_{21}& A_{22} \end{pmatrix},
$$
where $A_{11} \in \Mat_{kk}(S)$, $A_{12} \in \Mat_{k\, n - k}(S)$,
$A_{21} \in \Mat_{n - k\, k}(S)$, $A_{22} \in \Mat_{n - k\, n - k}(S)$,
$1 \leq k \leq n$, by defintion,
\begin{equation}
\label{A_Star}
   A^* = \begin{pmatrix}
   A^*_{11} \oplus A^*_{11} A_{12} D^* A_{21} A^*_{11} &
   \quad A^*_{11} A_{12} D^* \\[2ex]
   D^* A_{21} A^*_{11} &
   D^*
   \end{pmatrix},
\end{equation}
where $D = A_{22} \oplus A_{21} A^*_{11} A_{12}$. It can be proved that
this definition of $A^*$ implies that the equalities
$A^* = A^*A \oplus I=AA^*\oplus I$ are
satisfied,  and thus $A^*$ is a `regularized sum' of the series $I \oplus A
\oplus A^2 \oplus \dots$. Moreover, in the case when $A^*$ is defined as the least solution of $A^* = A^*A \oplus I$ and $A^* = AA^* \oplus I$, it can be shown
that it satisfies~\eqref{A_Star}.

Note that this recurrence
relation coincides with the formulas of escalator method of matrix
inversion in the traditional linear algebra over the field of real or
complex numbers, up to the algebraic operations used. Hence this algorithm
of matrix closure requires a polynomial number of operations in~$n$, see~\cite{Ca79,GM:10,LiMaRS2011,LRSS,LS-00,LS-01} for more details.

Let~$S$ be a positive semiring. The \emph{discrete stationary
Bellman equation or matrix Bellman equation} has the form
\begin{equation}
\label{AX}
	X = AX \oplus B,
\end{equation}
where $A \in \Mat_{nn}(S)$, $X, B \in \Mat_{ns}(S)$, and the matrix~$X$ is
unknown. Let $A^*$ be the closure of the matrix~$A$. It follows from the
identity $A^* = A^*A \oplus I$ that the matrix $A^*B$ satisfies this
equation. As in the scalar case, it
can be shown that for positive semirings under reasonable distributivity assumptions, if $A^*$ is defined as in~\eqref{cldef-gen} then $A^*B$ is the least in the set of solutions to equation~\eqref{AX}
with respect to the partial order in $\Mat_{ns}(S)$. Recall that in the
idempotent case
\begin{equation}
\label{aldef-matx}
A^*=\bigoplus_{i\geq 0} A^i=\sup_{i\geq 0} A^i.
\end{equation}

Consider also the case when $A=(a_{ij})$ is $n\times n$ {\em strictly upper-triangular}
(such that $a_{ij}=\0$ for $i\geq j$), or $n\times n$ {\em strictly lower-triangular}
(such that $a_{ij}=\0$ for $i\leq j$). In this case $A^n=O$, the all-zeros matrix,
and it can be shown by iterating $X=AX\oplus I$ that this equation has a unique
solution, namely
\begin{equation}
\label{a*nilp}
A^*=I\oplus A\oplus\ldots\oplus A^{n-1}.
\end{equation}
Curiously enough, formula~\eqref{a*nilp} works more generally in the
case of numerical idempotent semiring $\rmax$ (and other idempotent semirings):
in fact, the series~\eqref{aldef-matx}
converges there if and only if it can be truncated to~\eqref{a*nilp}. This
is closely related to the principal path interpretation of $A^*$ explained in
the next subsection.

\subsection{Weighted directed graphs and matrices over semirings}

Suppose that $S$ is a semiring with zero~$\0$ and unity~$\1$. It is well-known
that any square matrix $A = (a_{ij}) \in \Mat_{nn}(S)$ specifies a
\emph{weighted directed graph}. This geometrical construction includes
three kinds of objects: the set $X$ of $n$ elements $x_1, \dots, x_n$
called \emph{nodes}, the set $\Gamma$ of all ordered pairs $(x_i, x_j)$
such that $a_{ij} \neq \0$ called \emph{arcs}, and the mapping $A \colon
\Gamma \to S$ such that $A(x_i, x_j) = a_{ij}$. The elements $a_{ij}$ of
the semiring $S$ are called \emph{weights} of the arcs.

Conversely, any given weighted directed graph with $n$ nodes specifies a
unique matrix $A \in \Mat_{nn}(S)$.

This definition allows for some pairs of nodes to be disconnected if the
corresponding element of the matrix $A$ is $\0$ and for some channels to be
``loops'' with coincident ends if the matrix $A$ has nonzero diagonal
elements. This concept is convenient for analysis of parallel and
distributed computations and design of computing media and networks.

Recall that a sequence of nodes of the form
$$
	p = (y_0, y_1, \dots, y_k)
$$
with $k \geq 0$ and $(y_i, y_{i + 1}) \in \Gamma$, $i = 0, \dots, k -
1$, is called a \emph{path} of length $k$ connecting $y_0$ with $y_k$.
Denote the set of all such paths by $P_k(y_0,y_k)$. The weight $A(p)$ of a
path $p \in P_k(y_0,y_k)$ is defined to be the product of weights of arcs
connecting consecutive nodes of the path:
$$
	A(p) = A(y_0,y_1) \odot \cdots \odot A(y_{k - 1},y_k).
$$
By definition, for a `path' $p \in P_0(x_i,x_j)$ of length $k = 0$ the
weight is $\1$ if $i = j$ and $\0$ otherwise.

For each matrix $A \in \Mat_{nn}(S)$ define $A^0 = I = (\delta_{ij})$
(where $\delta_{ij} = \1$ if $i = j$ and $\delta_{ij} = \0$ otherwise) and
$A^k = AA^{k - 1}$, $k \geq 1$.  Let $a^{[k]}_{ij}$ be the $(i,j)$th
element of the matrix $A^k$. It is easily checked that
$$
   a^{[k]}_{ij} =
   \bigoplus_{\substack{i_0 = i,\, i_k = j\\
	1 \leq i_1, \ldots, i_{k - 1} \leq n}}
	a_{i_0i_1} \odot \dots \odot a_{i_{k - 1}i_k}.
$$
Thus $a^{[k]}_{ij}$ is the supremum of the set of weights corresponding to
all paths of length $k$ connecting the node $x_{i_0} = x_i$ with $x_{i_k} =
x_j$.

Let $A^*$ be defined as in~\eqref{aldef-matx}.
Denote the elements of the matrix $A^*$ by $a^{*}_{ij}$, $i, j = 1,
\dots, n$; then
$$
	a^{*}_{ij}
	= \bigoplus_{0 \leq k < \infty}
	\bigoplus_{p \in P_k(x_i, x_j)} A(p).
$$

The closure matrix $A^*$ solves the well-known \emph{algebraic path
problem}, which is formulated as follows: for each pair $(x_i,x_j)$
calculate the supremum of weights of all paths (of arbitrary length)
connecting node $x_i$ with node $x_j$. The closure operation in matrix
semirings has been studied extensively (see, e.g.,
\cite{BCOQ,Ca79,CG:79,Gol99,GM:79,GM:10,Leh-77,KoMa97,LS-01} and references therein).

\begin{example}[The shortest path problem]
{\rm Let $S = \rmin$, so the weights are real numbers. In this case
$$
	A(p) = A(y_0,y_1) + A(y_1,y_2) + \dots + A(y_{k - 1},y_k).
$$
If the element $a_{ij}$ specifies the length of the arc $(x_i,x_j)$ in some
metric, then $a^{*}_{ij}$ is the length of the shortest path connecting
$x_i$ with $x_j$.}
\end{example}

\begin{example}[The maximal path width problem]
{\rm Let $S = \rset \cup \{\0,\1\}$ with $\oplus = \max$, $\odot = \min$. Then
$$
	a^{*}_{ij} =
	\max_{p \in \bigcup\limits_{k \geq 1} P_k(x_i,x_j)} A(p),
	\quad
	A(p) = \min (A(y_0,y_1), \dots, A(y_{k - 1},y_k)).
$$
If the element $a_{ij}$ specifies the ``width'' of the arc
$(x_i,x_j)$, then the width of a path $p$ is defined as the minimal
width of its constituting arcs and the element $a^{*}_{ij}$ gives the
supremum of possible widths of all paths connecting $x_i$ with $x_j$.}
\end{example}

\begin{example}[A simple dynamic programming problem]
{\rm Let $S = \rmax$ and suppose $a_{ij}$ gives the \emph{profit} corresponding
to the transition from $x_i$ to $x_j$. Define the vector $B  = (b_i) \in
\Mat_{n1}(\rmax)$ whose element $b_i$ gives the \emph{terminal profit}
corresponding to exiting from the graph through the node $x_i$. Of course,
negative profits (or, rather, losses) are allowed. Let $m$ be the total
profit corresponding to a path $p \in P_k(x_i,x_j)$, i.e.
$$
	m = A(p) + b_j.
$$
Then it is easy to check that the supremum of profits that can be achieved
on paths of length $k$ beginning at the node $x_i$ is equal to $(A^kB)_i$
and the supremum of profits achievable without a restriction on the length
of a path equals $(A^*B)_i$.}
\end{example}

\begin{example}[The matrix inversion problem]
{\rm Note that in the formulas of this section we are using distributivity of
the multiplication $\odot$ with respect to the addition $\oplus$ but do not
use the idempotency axiom. Thus the algebraic path problem can be posed for
a nonidempotent semiring $S$ as well (this is well-known}). For
instance, if $S = \rset$, then
$$
	A^* = I + A + A^2 + \dotsb = (I - A)^{-1}.
$$
If $\|A\| > 1$ but the matrix $I - A$ is invertible, then this expression
defines a regularized sum of the divergent matrix power series
$\sum_{i \geq 0} A^i$.
\end{example}

We emphasize that this connection between the matrix closure operation and
solution to the Bellman equation gives rise to a number of different
algorithms for numerical calculation of the closure matrix. All these
algorithms are adaptations of the well-known algorithms of the traditional
computational linear algebra, such as the Gauss--Jordan elimination, various
iterative and escalator schemes, etc. This is a special case of the idempotent superposition principle.

In fact, the theory of the discrete stationary Bellman equation can be
developed using the identity $A^* = AA^* \oplus I$ as an additional axiom
without any substantive interpretation (the so-called \emph{closed
semirings}, see, e.g., \cite{Gol99,Leh-77}. Such theory can be based
on the following identities, true both for the case of idempotent
semirings with path interpretation, and the real numbers with
conventional arithmetic (assumed that $A$ and $B$ have appropriate sizes):
\begin{equation}
\label{e:conway}
\begin{split}
(A\oplus B)^*&=(A^*B)^*A^*,\\
(AB)^*A&=A(BA)^*.
\end{split}
\end{equation}

\subsection{Basic operations}

Let $S$ be an idempotent or positive semiring. Then $S$ is a partial ordered set (or poset) with respect to the canonical order.
Suppose that $S$ is a lattice, i.e. for each pair of elements $x,y$, there exists
the least lower bound $x\vee y$ called {\rm supremum} and the greatest lower bound
$x\wedge y$ called {\rm infimum}. See details in Subsection 9.1 below. For basic numerical semirings (see Subsection 7.2 above) these operations are maximum and minimum.

We shall say that the semiring operations, supremum, infinum and the unary closure operation (Kleene star-operation) are {\it basic operations}.

We shall say that $S$ is a {\it completed semifield} if $S$ is a complete semiring and $S$ without the element $\sup S$ is a semifield. Then the
(unary) inversion  $x\mapsto x^{-1}$ is obviously well defined for every element. If $S$ is a semifield or completed semifield, then we shall say that the inversion operation is also basic.

For basic numerical positive semirings all the basic operations are very simple and easy for computer implementations.

\section{Algorithms of tropical/idempotent mathematics}

\subsection{The correspondence principle for computations}

Of course, the idempotent correspondence principle is valid for
algorithms as well as for their software and hardware implementations~\cite{LiMa95,LiMa96,LiMa98,LMR-00,LiMaRS2011,LRSS}.
Thus:

\begin{quote}
{\it If we have an important and interesting numerical algorithm, then
there is a good chance that its semiring analogs are important and
interesting as well.}
\end{quote}

In particular, according to the superposition principle,
analogs of linear
algebra algorithms are especially important. Note that
numerical algorithms
for standard infinite-dimensional linear problems over idempotent
semirings (i.e., for
problems related to idempotent integration, integral operators and
transformations, the Hamilton-Jacobi and generalized Bellman equations)
deal with the corresponding finite-dimensional (or finite) ``linear
approximations''. Nonlinear algorithms often can be approximated by linear
ones. Thus the idempotent linear algebra is a basis for the idempotent
numerical analysis.

Moreover, it is well-known that linear algebra algorithms easily lend themselves to parallel computation; their idempotent analogs admit
parallelization as
well. Thus we obtain a systematic way of applying parallel computing to
optimization problems. In this paper we do not deal with parallel algorithms and their implementations.


\subsection{Universal algorithms}

Computational algorithms are constructed
on the basis of certain primitive operations. These operations manipulate
data that describe ``numbers.'' These ``numbers'' are elements of a
``numerical domain,'' i.e., a mathematical object such as the field of
real numbers, the ring of integers, numerical and idempotent semirings and semifields.

In practice, elements of the numerical domains are replaced
by their computer representations, i.e., by elements of certain finite
models of these domains. Examples of models that can be conveniently used
for computer representation of real numbers are provided by various
modifications of floating point arithmetics, approximate arithmetics of
rational numbers~\cite{LRTch-08}, and interval arithmetics. The difference
between mathematical objects (``ideal'' numbers) and their finite
models (computer representations) results in computational (e.g.,
rounding) errors.

An algorithm is called {\it universal\/} if it is independent of a
particular numerical domain and/or its computer representation.
A typical example of a universal algorithm is the computation of the
scalar product $(x,y)$ of two vectors $x=(x_1,\dots,x_n)$ and
$y=(y_1,\dots,y_n)$ by the formula $(x,y)=x_1y_1+\dots+x_ny_n$.
This algorithm (formula) is independent of a particular
domain and its computer implementation, since the formula is
well-defined for any semiring. It is clear that one algorithm can be
more universal than another. For example, the simplest Newton--Cotes formula, the
rectangular rule, provides the most universal algorithm for
numerical integration. In particular, this formula is valid also for
idempotent integration (that is, over any idempotent semiring, see  e.g.
\cite{KoMa97,Li2005,Mas86,Mas87a,Mas87b}.
Other quadrature formulas (e.g., combined trapezoid rule or the Simpson
formula) are independent of computer arithmetics and can be
used (e.g., in the iterative form) for computations with
arbitrary accuracy. In contrast, algorithms based on
Gauss--Jacobi formulas are designed for fixed accuracy computations:
they include constants (coefficients and nodes of these formulas)
defined with fixed accuracy. (Certainly, algorithms of this type can
be made more universal by including procedures for computing the
constants; however, this results in an unjustified complication of the
algorithms.)

Modern achievements in software development
and mathematics make us consider numerical algorithms and their
classification from a new point of view. Conventional numerical
algorithms are oriented to software (or hardware) implementation based
on floating point arithmetic and fixed accuracy. However,
it is often desirable to perform computations with variable (and
arbitrary) accuracy. For this purpose, algorithms are required
that are independent of the accuracy of computation and of the
specific computer representation of numbers. In fact, many
algorithms are independent not only of the computer representation
of numbers, but also of concrete mathematical (algebraic) operations
on data. In this case, operations themselves may be considered as variables.
Such algorithms are implemented in the form of {\it generic programs} based on
abstract data types that are
defined by the user in addition to the predefined types provided by the
language. The corresponding program tools appeared as
early as in Simula-67, but modern object-oriented languages (like
$C^{++}$, see, e.g., \cite{Lor:93,Pohl:97}) are more convenient for generic programming. Computer algebra algorithms used in such systems as Mathematica,
Maple, REDUCE, and others are also highly universal.

A different form of universality is featured by
iterative algorithms (beginning with the successive approximation
method) for solving differential equations (e.g., methods of
Euler, Euler--Cauchy, Runge--Kutta, Adams, a number of important
versions of the difference approximation method, and the like),
methods for calculating elementary and some special functions based on
the expansion in Taylor's series and continuous fractions
(Pad\'e approximations). These algorithms are independent of the computer
representation of numbers.

The concept of a generic program was introduced by many authors;
for example, in~\cite{Leh-77} such programs were called `program schemes.'
In this paper, we discuss  universal algorithms implemented in the form of generic
programs.

\subsection{The correspondence principle for hardware design}

A systematic application of the correspondence principle to computer
calculations leads to a unifying approach to software and hardware
design.

The most important and standard numerical algorithms have many hardware
realizations in the form of technical devices or special processors.
{\it These devices often can be used as prototypes for new hardware
units generated by substitution of the usual arithmetic operations
for its semiring analogs and by addition tools for performing neutral
elements $\0$ and} $\1$ (the latter usually is not difficult). Of course,
the case of numerical semirings consisting of real numbers (maybe except
neutral elements)  and semirings of numerical intervals is the most simple and natural .
Note that for semifields (including $\rmax$ and $\rmin$)
the operation of division is also defined.

Good and efficient technical ideas and decisions can be transposed
from prototypes into new hardware units. Thus the correspondence
principle generated a regular heuristic method for hardware design.
Note that to get a patent it is necessary to present the so-called
`invention formula', that is to indicate a prototype for the suggested
device and the difference between these devices~\cite{LiMa95,LiMa96,LiMa98,LMR-00,LiMaRS2011,LRSS}.

Consider (as a typical example) the most popular and important algorithm
of computing the scalar product of two vectors:
\begin{equation}
(x,y)=x_1y_1+x_2y_2+\cdots + x_ny_n.
\end{equation}
The universal version of (12) for any semiring $A$ is obvious:
\begin{equation}
(x,y)=(x_1\odot y_1)\oplus(x_2\odot y_2)\oplus\cdots\oplus
(x_n\odot y_n).
\end{equation}
In the case $A=\rmax$ this formula turns into the following one:
\begin{equation}
(x,y)=\max\{ x_1+y_1,x_2+y_2, \cdots, x_n+y_n\}.
\end{equation}

This calculation is standard for many optimization algorithms, so
it is useful to construct a hardware unit for computing (14). There
are many different devices (and patents) for computing (12) and every
such device can be used as a prototype to construct a new device for
computing (14) and even (13). Many processors for matrix multiplication
and for other algorithms of linear algebra are based on computing
scalar products and on the corresponding ``elementary'' devices
respectively, etc.


\subsection{The correspondence principle for software design}

Software implementations for universal semiring algorithms are not
as efficient as hardware ones (with respect to the computation speed)
but they are much more flexible. Program modules can deal with abstract (and
variable) operations and data types. Concrete values for these
operations and data types can be defined by the corresponding
input data. In this case concrete operations and data types are generated
by means of additional program modules. For programs written in
this manner it is convenient to use special techniques of the
so-called object oriented (and functional) design, see, e.g.,~\cite{Lor:93,Pohl:97}. Fortunately, powerful tools supporting the
object-oriented software design have recently appeared including compilers
for real and convenient programming languages (e.g. $C^{++}$ and Java) and modern computer algebra systems.

Recently, this type of programming technique has been dubbed the so-called
generic programming (see, e.g., \cite{SL:94}). To help automate the
generic programming, the so-called Standard Template Library (STL)
was developed in the framework of $C^{++}$ \cite{Pohl:97,SL:94}. However,
high-level tools, such as STL, possess both obvious advantages
and some disadvantages and must be used with caution.

\subsection{Complexity of algorithms in idempotent mathematics}

We shall use the well known standard terminology of the complexity theory (time complexity, space complexity, asymptotic computational complexity,
polynomial complexity, NP-hard problems etc.), see, e.g., the corresponding Wikipedia articles. The time complexity of an algorithm quantifies the
amount of time taken by this algorithm to run as a function of the size of the input to the problem. The time complexity of an algorithm
is commonly expressed using big O notation, which suppresses multiplicative constants and lower order terms. When expressed this way, the time complexity
is said to be described asymptotically, i.e., as the input size goes to infinity. The time complexity is commonly estimated by counting the number of elementary operations performed by the algorithm. {\it In idempotent mathematics (and mathematics over positive semirings) the elementary operations are basic
operations described in Subsection 7.4 above}.

For the space complexty, the situation is quite similar. The following definition is important for us. We shall say that two algorithms have the {\it same complexity} if they have the same asymptotic time and space complexity.

In principle idempotent mathematcs and its algorithms are more simple with respect to traditional mathematics. That is why the most important algorithms
of idempotent mathematics (and espcially idempotent linear algebra) are polynomial. For example, many algorithms of solving the stationary discrete (matrix)
Bellman equations have the complexity $O(n^3)$, see~\cite{Ca71,Ca79,GM:79,GM:10,LiMaRS2011,LMa-00,LS-00,LS-01,LRSS}. Many other polynomial
algorithms of linear algebra are examined in~\cite{BCOQ,But:10,BZ,CG:79,Gr,GP,Har+09,LMRS,T}. However, NP-hard problems exist (e.g., in tropical and idempotent
linear algebra), see~\cite{Gr,GP,LMRS,Sh}. Some frontiers of polynomial computations in tropical geometry are investigated in~\cite{T}.

\section{Interval analysis in idempotent mathematics and complexity of algorithms}

\subsection{Interval extensions of algorithms}

Interval analysis appears for treating input and output data with  uncertainties (interval data).
Traditional interval analysis is a nontrivial and popular mathematical area, see, e.g.,~\cite{AH:83,Fie+06,Kre+98,Moo:79,Neu:90}. An ``idempotent'' version of interval analysis (and moreover interval analysis over positive semirings) appeared in~\cite{LS-00,LS-01,Sob-99}. Rather many publications on the subject appeared later, see, e.g.,~\cite{CC-02,Fie+06,Har+09,Mys-05,Mys-06}. Interval analysis over the positive semiring $\textbf{R}_+$ was discussed in~\cite{BN-74}.
In the framework of idempotent mathematics, interval analysis gives exact interval solutions without any conditions of smallness on uncertainty intervals.

Let a set~$S$ be partially ordered by a relation $\preceq$. Below (in Subsection 9.2 partially ordered sets (or posets for the sake of brevity)
will be discussed in details.
A \emph{closed interval} in~$S$ is a subset of the form $\x = [\podx, \nadx] =
\{\, x \in S \mid \podx \preceq x \preceq \nadx\, \}$, where the elements $\podx \preceq
\nadx$ are called \emph{lower} and \emph{upper bounds} of the interval $\x$.
The order~$\preceq$ induces a partial ordering on the set of all closed
intervals in~$S$: $\x \preceq \y$ iff $\podx \preceq \pody$ and $\nadx \preceq \nady$.

A \emph{weak interval extension} $I(S)$ of a positive semiring~$S$ is the
set of all closed intervals in~$S$ endowed with operations $\oplus$
and~$\odot$ defined as ${\x \oplus \y} = [{\podx \oplus \pody}, {\nadx \oplus
\nady}]$, ${\x \odot \y} = [{\podx \odot \pody}, {\nadx \odot \nady}]$ and a partial
order induced by the order in $S$. The closure operation in $I(S)$ is
defined by $\x^* = [\podx^*, \nadx^*]$. There are some other interval extensions (including the so-called strong interval extension~\cite{LS-01}) but the weak extension is more convenient.

The extension $I(S)$ is a positive semiring; $I(S)$ is idempotent if $S$ is an idempotent semiring.
A universal algorithm over $S$ can be applied to $I(S)$ and we shall get an interval version of the initial algorithm. However, there are some conditions
for interval extensions of algorithms to be sure that these extensions are good enough, see below.
Usually both the versions have the same complexity. For the discrete stationary Bellman equation and the corresponding optimization problems on graphs, interval analysis was examined in~\cite{LS-00,LS-01} in details. Other problems of idempotent linear algebra  were examined in~\cite{CC-02,Fie+06,Har+09,Mys-05,Mys-06}.

Idempotent mathematics appears to be remarkably simpler than its
traditional analog. For example, in traditional interval arithmetic,
multiplication of intervals is not distributive with respect to addition of
intervals, whereas in idempotent interval arithmetic this distributivity is
preserved. Moreover, in traditional interval analysis the set of all
square interval matrices of a given order does not form even a semigroup
with respect to matrix multiplication: this operation is not associative
since distributivity is lost in the traditional interval arithmetic. On the
contrary, in the idempotent (and positive) case associativity is preserved. Finally, in
traditional interval analysis some problems of linear algebra, such as
solution of a linear system of interval equations, can be very difficult
(more precisely, they are $NP$-hard, see~
\cite{Cox-99,Fie+06,Kre+98} and references therein). It was noticed  in~\cite{LS-00,LS-01} that in the idempotent case solving an interval linear system
requires a polynomial number of operations (similarly to the usual Gauss
elimination algorithm).  Two properties that make the idempotent interval
arithmetic so simple are monotonicity of arithmetic operations and
positivity of all elements of an idempotent semiring.

Usually interval estimates in idempotent mathematics are exact. In the traditional theory such estimates tend to be overly pessimistic.

\subsection{Intervals in partially ordered sets and interval regular mappings}

Let us start with some basic notions of the theory of lattices and partially ordered sets.
The reader is referred to \cite{Bir} for more information.

\begin{definition}
Binary relation $\preceq$ on a set $S$ is called a {\rm partial order} if it satisfies the following
axioms: 1) $a\leq a$, 2) $a\preceq b$ and $b\preceq a$ imply $a=b$, 3)
$a\preceq b$ and $b\preceq c$ imply $a\preceq c$. In this case $S$ is called a
{\rm partially ordered set} or, briefly, a {\rm poset}.
\end{definition}

If $S$ is a Cartesian product $S_1\times S_2$ where $S_1$ and $S_2$ are posets ordered
with $\preceq_1$ and $\preceq_2$ respectively,
one can naturally introduce relation $\preceq$ on $S$, by
$(x_1,y_1)\preceq (x_2,y_2) \Leftrightarrow x_1\preceq_1 x_2$ and $y_1\preceq_2 y_2$.

For the dynamics, consider mappings of partially ordered sets. A mapping
$\phi\colon S\to T$ is a morphism of partially ordered sets $S$ and $T$ if
$x\preceq y$ implies $\phi(x)\preceq\phi(y)$. This mapping is an isomorphism if it is
a bijection (i.e. one-to-one correspondence) between $S$ and $T$. Also note that if
$\phi_1$ is a morphism between $S_1$ and $T_1$, and $\phi_2$ is a morphism (resp.
isomorphism) between
$S_2$ and $T_2$, then mapping $\phi_1\times\phi_2\colon (x,y)\mapsto (\phi_1(x),\phi_2(y))$
 is a morphism (resp. isomorphism) between $S_1\times T_1$ and $S_2\times T_2$.

For a poset $S$ and a subset $X\subseteq S$,
an element $t\in S$ is called an {\em upper bound} (resp. a {\em lower bound})
of $X$ if $t\geq x$ (resp. $t\leq x$) for every $x\in X$.
\begin{definition}
A poset $S$ is called a lattice if for each pair of elements $x,y$, there exists
the least lower bound $x\vee y$ called {\rm supremum} and the greatest lower bound
$x\wedge y$ called {\rm infimum}.
\end{definition}

A lattice $S$ is {\em complete} if every subset $H\subseteq S$ (not necessarily finite) has supremum
and infimum in $S$, and it is {\em conditionally complete} if supremum exists for
each subset bounded from above, and infimum exists for each subset bounded from below.
Note that the last two statements are equivalent so that, formally, only one of them is needed.

\begin{example}
{\rm
Consider
the set of natural numbers $N$ ordered in such a way that $n_1\preceq n_2$ if and only if $n_1$ divides $n_2$. Then it can be verified
that $m\vee n$ is the least upper bound of $m$ and $n$ while $m\wedge n$ is their
greatest common divisor. Other related examples are the lattice of subsets ordered by
inclusion where $\wedge=\cap$ and $\vee=\cup$, or the lattice of convex sets ordered by inclusion
where $\wedge=\cap$ but $\vee$ is the convex hull of the arguments. However we emphasize that
such examples are not important to us here.}
\end{example}

\begin{example}
{\rm Any {\em linearly ordered} set $S$, i.e. such that for each $x,y\in S$ there is $x\preceq y$
or $y\preceq x$, is a lattice where both $x\vee y\in\{x,y\}$ and $x\wedge y\in\{x,y\}$.
For example we can take the real line $\R$ or any subset of the real line, e.g., an interval
$[a,b]$ or a set of nonnegative numbers $\R_+$. However, $\R$ and $\R_+$ are only conditionally
complete, and for their completion we can consider $\Hat{R}:=\R\cup\{-\infty\}\cup\{\infty\}$
and $\Hat{\R_+}:=\R_+\cup\{\infty\}$.}
\end{example}

From this one can construct slightly more
complicated examples, which will not be linearly ordered, e.g.,
by means of Cartesian products.
The following lattice-theoretic definition is of particular importance to us.

\begin{definition}
Let $S$ be a set partially ordered by a relation $\preceq$. A {\rm (closed) interval} in $S$ is
a subset of the form $\intx:=[\podx,\nadx]=\{t\in S\colon \podx\preceq t\preceq\nadx\}$,
where $\podx,\nadx\in S$ and $\podx\preceq\nadx$. The elements $\podx$ and $\nadx$
are called {\rm lower bound} and {\rm upper bound} of $x$ respectively.
\end{definition}
Intervals can be viewed as pairs of lower and upper bounds. The set of such pairs
is denoted by $I(S)$ and called {\em interval extension} of $S$.
It is evident that $I(S)\subseteq S\times S$ (but $I(S)\neq S\times S$) and that $I(S)$ is a poset.
It can be shown that
$I(S\times T)=I(S)\times I(T)$ for posets $S$ and $T$.

Consider an algorithm $\cA$ on posets. It takes input data $(x_1,\ldots, x_n)$ and
generates output $(y_1,\ldots, y_m)$. Here $x_i\in S_i$ for $i=1,\ldots,n$ and $y_j\in T_j$ for
$j=1,\ldots,m$, where $S_i$ and $T_j$ are posets. Hence, this algorithm induces a mapping
$\phi_{\alg}\colon S\to T$ where $S:=S_1\times\cdots\times S_n$ and
$T:=T_1\times\cdots\times T_m$. We call an algorithm $\cA$ {\em positive} or
{\em nondecreasing} if the corresponding mapping
$S_1\times\cdots\times S_n\to T_1\times\cdots\times T_m$ is nondecreasing.

\begin{proposition}
\label{p:main1}
If an algorithm $\cA$ is positive, then applied to the lower and upper
bounds of an interval of
$S_1\times\cdots\times S_n$ where $S_1,\ldots,S_n$ are posets,
 it yields exact interval
bounds on the application of $\alg$ to the whole interval.
\end{proposition}
\begin{proof}
These properties follow immediately from the positivity of $\alg$.
\end{proof}

It is clear that such algorithm and its interval extension have the same complexity.
Using Proposition~\ref{p:main1}, the interval extension $I(\cA)$
can be defined on the bounds of
intervals only, and the result of $I(\cA)$ is a mapping from $I(S_1)\times\cdots\times I(S_n)$
to $I(T_1)\times\cdots\times I(T_m)$. The complexity of $I(\cA)$ only doubles the complexity of $\cA$.

We proceed with the following observations.

\begin{proposition}
Cartesian product of positive algorithms is a positive algorithmm.
\end{proposition}

\begin{proposition}
Let $S$ be a complete or conditionally complete poset. Operations
$(x,y)\mapsto x\vee y$ and $(x,y)\mapsto x\wedge y$ are positive mappings
$S\times S\to S$.
\end{proposition}

\begin{definition}
\label{def:intreg}
A mapping $f: S\to T$ will be called interval regular or, briefly, $i$-regular if
the following condition hold: for any interval $\intx\subseteq S$ there exists
$\inty\subseteq T$ such that $f(\intx)\subseteq\inty$, $\pody\in f(\intx)$ and
$\nady\in f(\intx)$.
\end{definition}

For example, consider $\Hat{\R}_+:=\Rp\cup\{+\infty\}$ and unary operation
$x\mapsto x^-$ defined as $x^-:= x^{-1}$ on finite numbers, $0^-=+\infty$ and
$\infty^-:=0$. This mapping is not positive, but $i$-regular. Observe that an
interval $[a,b]$ is mapped to $[b^{-1},a^{-1}]$.

Any algorithm $\cA$ induces a mapping $f_{\cA}$, hence we can define $i$-regular algorithms.

\begin{definition}
\label{def:algointreg}
An algorithm $\cA$ is called {\em $i$-regular\/} if
$f_{\cA}$ is $i$-regular.
\end{definition}

\begin{definition}
\label{def:cintreg}
An algorithm $\cA$ is called {\em $ci$-regular}
if $f_{\cA}$ is interval regular and this algorithm and its interval extension have the same (asymptotic) complexity.
\end{definition}

For example, the inversion operations in semifields and completed semifields are $ci$-regular but not positive. The following simple and obvious  proposition illustrates
a typical application of the notion of $ci$-regular algorithms.

\begin{proposition}
\label{p:main2}
Any composition or Cartesian product of $ci$-regular algorithms is a $ci$-regular algorithm.
\end{proposition}

\begin{corollary}
Any positive algorithm is $ci$-regular. The complexity of its interval
version is the same.
\end{corollary}

\subsection{Interval analysis over a fixed basic semiring}

Fix a basic positive semiring $K$, e.g., a basic numerical semiring in the sense of Subsection 7.1.
For the sake of simplicity, we suppose that $K$ is complete.

So input output data run Cartesian products of several copies of $K$.
Note that Cartesian product of positive semirings $S_1\times S_2$ is a positive semiring, with
respect to the Cartesian product of orders in $S_1$ and $S_2$.

\begin{definition}
An algorithm $\cA$ is called {\em elementary} if it can be realized as a composition or Cartesian product of a finite number of basic operations
from $K$, see Subsection 7.4 above.
\end{definition}

Algorithms which are not elementary might use if-else constructions, however,
they may be still $ci$-regular.

From the results presented in Subsection 9.2 we can easily deduce the following result.

\begin{theorem}
The basic unary operation $x\mapsto x^*$ induces a nondecreasing mapping $S\to S$,
and the basic binary operations $\oplus,\odot$, as well as supremum and infimum, induce nondecreasing mappings $S\times S\to S$.
If an elementary algorithm $\cP$ uses only these basic operations over a positive semiring, then it is positive.
If $K$ is a completed semifield, then the unary inversion operation is $ci$-regular but not positive. Every composition of
elementary (and other $ci$-regular algorithms) is a $ci$-regular algorithm.
\end{theorem}

\begin{example}
There exists an algorithm for computing $A^*$ which can be represented
as a composition of positive operations with respect to $K$, see Subsection 7.2 above. Hence this algorithm is positive.
Moreover each standard algorithm solving the stationary discrete Bellman equation (see, e.g. \cite{LRSS}) is positive and
$ci$-regular.
\end{example}

Some standard algorithms for solving the matrix equations $Ax=b,  Ax=Bx, Ax=By$, and other problems of tropical/idempotent
linear algebra use the so-called (binary) pseudodivision operation, see, e.g. \cite{BCOQ,Har+09}. In principle, this operation
is not elementary or $ci$-regular. However, using results presented in~\cite{Har+09}, it is possible to reduce this operation to
$ci$-regular algorithms for all the basic numerical semirings.  So the corresponding algorithms of linear algebra are $ci$-regular.

\section*{Acknowledgments}

The author is grateful to S.~N.~Sergeev, A.~N.~Sobolevski and T.~Nowak for their kind help.

\end{document}